\documentclass{amsart}

\usepackage{latexsym,amsfonts} 
\usepackage{amsmath,amssymb,graphics,setspace}
\usepackage{amsthm}
\usepackage{MnSymbol}
\usepackage{mathrsfs}

\usepackage{marvosym}
\usepackage{paralist}
\usepackage{color}

\usepackage{graphicx}
\usepackage{fontenc}%

\usepackage[verbose]{wrapfig}

\newtheorem{theorem}{Theorem}
\newtheorem{lemma}{Lemma}

\newtheorem{remark}{Remark}
\theoremstyle{definition}
\newtheorem{definition}{Definition}

\newtheorem{example}{Example}

\newcommand{\norm}[1]{\left\|#1\right\|}
\newcommand{\cv}{\mbox{conv}} 
\newcommand{\scl}{\mathop{\mbox{cl}}\limits^{\doublewedge}} 
\newcommand{\dnear}{\delta_{\Phi}} 
\newcommand{\dcap}{\mathop{\cap}\limits_{\Phi}} 
\newcommand{\sn}{\mathop{\delta}\limits^{\doublewedge}} 
\newcommand{\snd}{\mathop{\delta_{_{\Phi}}}\limits^{\doublewedge}} 
\newcommand{\Cn}{\mbox{\large$\mathfrak{C}$}} 
\newcommand{\Cmn}{\mbox{max}\mathfrak{C}} 
\newcommand{\Cdn}{\mathfrak{C}_{\Phi}}
\newcommand{\Cmdn}{\mbox{max}\mathfrak{C}_{\Phi}}
\newcommand{\Csnd}{{\mathop{\mathfrak{C}}\limits^{\doublewedge}}_{\Phi}} 

\begin{document}

\title{Strongly Near Vorono\"{i} Nucleus Clusters}

\author[J.F. Peters]{J.F. Peters$^{\alpha}$}
\email{James.Peters3@umanitoba.ca, einan@adiyaman.edu.tr}
\address{\llap{$^{\alpha}$\,}Computational Intelligence Laboratory,
University of Manitoba, WPG, MB, R3T 5V6, Canada and
Department of Mathematics, Faculty of Arts and Sciences, Ad\.{i}yaman University, 02040 Ad\.{i}yaman, Turkey}
\author[E. Inan]{E. \.{I}nan$^{\beta}$}
\address{\llap{$^{\beta}$\,} Department of Mathematics, Faculty of Arts and Sciences, Ad\i yaman University, 02040 Ad\i yaman, Turkey and Computational Intelligence Laboratory,
University of Manitoba, WPG, MB, R3T 5V6, Canada}
\thanks{The research has been supported by the Scientific and Technological Research
Council of Turkey (T\"{U}B\.{I}TAK) Scientific Human Resources Development
(BIDEB) under grant no: 2221-1059B211402463 and the Natural Sciences \&
Engineering Research Council of Canada (NSERC) discovery grant 185986.}

\subjclass[2010]{Primary 54E05 (Proximity); Secondary 62H30 (Cluster Analysis), 68T10 (Pattern Recognition)}

\date{}

\dedicatory{Dedicated to the Memory of Som Naimpally}

\begin{abstract}
This paper introduces nucleus clustering in Vorono\"{i} tessellations of plane surfaces with applications in the geometry of digital images.  A \emph{nucleus cluster} is a collection of Vorono\"{i} regions that are adjacent to a Vorono\"{i} region called the cluster nucleus. Nucleus clustering is a carried out in a strong proximity space.  Of particular interest is the presence of maximal nucleus clusters in a tessellation.  Among all of the possible nucleus clusters in a Vorono\"{i} tessellation, clusters with the highest number of adjacent polygons are called \emph{maximal nucleus clusters}. The main results in this paper are that strongly near nucleus clusters are strongly descriptively near and every collection of Vorono\"{i} regions in a tessellation of a plane surface is a Zelins'kyi-Soltan-Kay-Womble convexity structure.
\end{abstract}

\keywords{Convexity Structure, Nucleus Clustering, Strong Proximity, Vorono\"{i} Tessellation}

\maketitle

\section{Introduction}
This paper introduces nucleus clustering in Vorono\"{i} tessellations of surfaces in Euclidean space $\mathbb{R}^d, d\geq 2$.  In this article, nucleus clustering is restricted to plane surfaces with applications in the geometry of digital images.  A \emph{nucleus cluster} is a collection of Vorono\"{i} regions that are adjacent to a Vorono\"{i} region called the cluster nucleus, which is a variation of the notion of a Harer-Edelsbrunner nerve~\cite[\S III.2, p. 59]{Edelsbrunner2010compTop}.

\setlength{\intextsep}{0pt}
\begin{wrapfigure}[9]{R}{0.35\textwidth}
\begin{minipage}{5.2 cm}
\centering
\includegraphics[width=25mm]{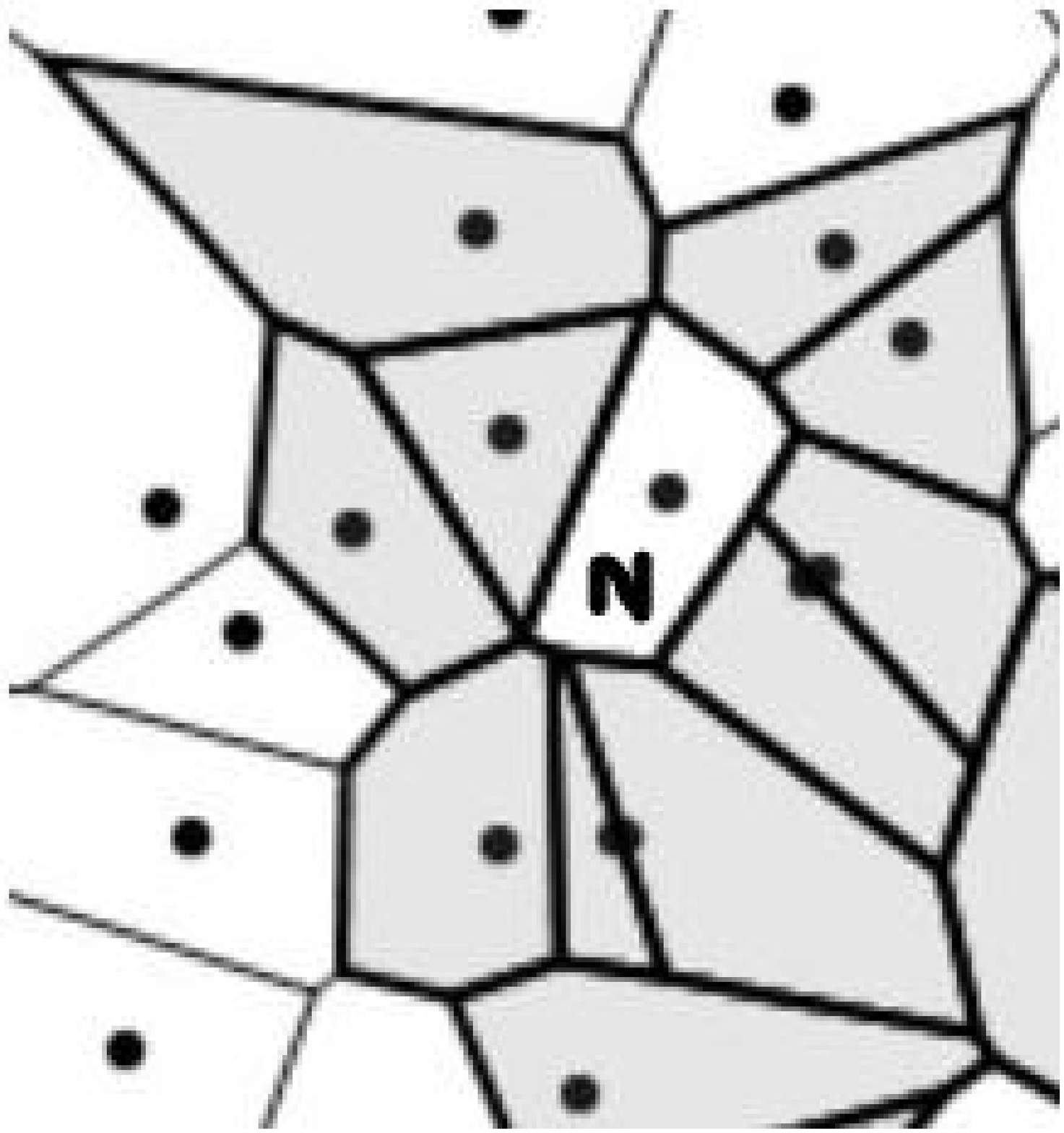}
\caption[]{$\mbox{}$\\ Nucleus Cluster}
\label{fig:nucleusCluster}
\end{minipage}
\end{wrapfigure}

Every Vorono\"{i} region of a site $s$ is a convex polygon containing all points that are nearer $s$ than to any other site in a Vorono\"{i} tessellation of a surface. Vorono\"{i} regions are strongly near, provided the regions have points in common.  This form of clustering leads to the introduction of what are known as nucleus-clusters.  A \emph{nucleus cluster} is a collection of Vorono\"{i} regions that are strongly near a central Vorono\"{i} region called the cluster nucleus in a Vorono\"{i} tessellation.  A \emph{maximal nucleus cluster} is a collection of a maximal number of Vorono\"{i} regions that are strongly near the mesh nucleus. Maximal nucleus clusters serve as indicators of high object concentration  in a tessellated image.  This form of clustering leads to object recognition in many forms of application images. 

\section{Preliminaries}
This section introduces strongly near proximity and Vorono\"{i} tessellation of a plane surface based on recent work on computational proximity~\cite{Peters2016ComputationalProximity}, computational geometry~\cite{Edelsbrunner1999,Edelsbrunner2001,Edelsbrunner2014,Edelsbrunner2010compTop}.   Strong proximities were introduced in~\cite{Peters2015AMSJmanifolds}, elaborated in~\cite{Peters2016ComputationalProximity} (see, also,~\cite{Inan2015}) and are a direct result of earlier work on proximities~\cite{DiConcilio2006,DiConcilio2013mcs,Naimpally1970,Naimpally2009,Naimpally2013}.  Nonempty sets $A$ and $B$ have strong proximity (denoted $A\ \sn\ B$), provided $A$ and $B$ have points in common.  Let $E$ be the Euclidean plane, $S\subset E$ (set of mesh generating points), $s\in S$.
A Vorono\"{i} region (denoted by $V(s)$) is defined by
\[
V(s) = \left\{x\in E: \norm{x - s}\leq \norm{x - q}, \mbox{for all}\ q\in S\right\}\ \mbox{(Vorono\"{i} region)}.
\]


\begin{example}
A partial view of a Vorono\"{i} tessellation of a plane surface is shown in Fig.~\ref{fig:nucleusCluster}.  The Vorono\"{i} region $N$ in this tessellation is the nucleus of a mesh clustering containing all of those polygons adjacent to $N$.  Let $X$ be a collection of Vorono\"{i} regions containing $N$, endowed with the strong proximity $\sn$.  Briefly, a proximity relation is strong, provided $A\ \sn\ B, A,B\in X$ have points in common.  
Then the nucleus mesh cluster (denoted by $\Cn\ N$) in this sample tessellation is defined by
\[
\Cn\ N = \left\{A\in X: \scl\ A\ \sn\ N\right\}\ \mbox{(Vorono\"{i} mesh nucleus cluster)}.\\
\]
That is, a nucleus mesh cluster $\Cn\ N$ is a collection of nonempty sets $A$ whose closure is strongly near the cluster nucleus $N$ (in that case, each $A\in \Cn\ N$ has points in common with $N$).  For example, the set of points in the convex polygon $N$ in Fig.~\ref{fig:nucleusCluster} has points in common with each of the adjacent polygons, {\em i.e.}, each polygon adjacent to $N$ has an edge in common with $N$.  Let $B$ be a polygon adjacent to $N$.  $B\ \sn\ N$, since $B$ and $N$ have in edge in common.  
\qquad \textcolor{blue}{\Squaresteel}
\end{example}

A \emph{concrete} (\emph{physical}) set $A$ of points $p$ that are described by their location and physical characteristics, {\em e.g.}, gradient orientation (angle of the tangent to $p$.  Let $\varphi(p)$ be the gradient orientation of $p$.   For example, each point $p$ with coordinates $(x,y)$ in the concrete subset $A$ in the Euclidean plane is described by a feature vector of the form $(x,y,\varphi(p(x,y))$. Nonempty concrete sets $A$ and $B$ have descriptive strong proximity (denoted $A\ \snd\ B$), provided $A$ and $B$ have points with matching descriptions.
In a region-based, descriptive proximity extends to both abstract and concrete sets~\cite[\S 1.2]{Peters2016ComputationalProximity}.  For example, every subset $A$ in the Euclidean plane has features such as area and diameter.  Let $(x,y)$ be the coordinates of the centroid $m$ of $A$.  Then $A$ is described by feature vector of the form $(x,y,area, diameter)$.  Then regions $A,B$ have descriptive proximity (denoted $A\ \snd\ B$), provided $A$ and $B$ have matching descriptions.

The notion of strongly proximal regions extends to convex sets.  A nonempty set $A$ is a \emph{convex set} (denoted $\cv A$), provided, for any pair of points $x,y\in A$, the line segment $\overline{xy}$ is also in $A$.  The empty set $\emptyset$ and a one-element set $\left\{x\right\}$ are convex by definition.  Let $\mathscr{F}$ be a family of convex sets.  From the fact that the intersection of any two convex sets is convex~\cite[\S 2.1, Lemma A]{Edelsbrunner2014}, it follows that
\[
\mathop{\bigcap}\limits_{A\in\mathscr{F}} A\ \mbox{is a convex set}.
\]
Convex sets $\cv A, \cv B$ are strongly proximal (denote $\cv A \sn\ \cv B$), provided $\cv A, \cv B$ have points in common.  Convex sets $\cv A, \cv B$ are descriptively strongly proximal (denoted $\cv A \snd\ \cv B$), provided $\cv A, \cv B$ have matching descriptions.


Let $X$ be a Vorono\"{i} tessellation of a plane surface equipped with the strong proximity $\sn$ and descriptive strong proximity $\snd$ and let $A,N\in X$ be Vorono\"{i} regions.   The pair $\left(X,\left\{\sn,\snd\right\}\right)$ is an example of a proximal relator space~\cite{Peters2016relator}.
The two forms of nucleus clusters (ordinary nucleus cluster denoted by $\Cn$) and descriptive nucleus clusters are examples of mesh nerves~\cite[\S 1.10, pp. 29ff]{Peters2016ComputationalProximity}, defined by
\begin{align*}
\Cn N &= \left\{A\in X: A\ \sn\ N\right\}\ \mbox{(nucleus cluster)}.\\
\Cdn N  &= \left\{A\in X: A\ \snd\ N\right\}\ \mbox{(descriptive nucleus cluster)}.
\end{align*}

A nucleus cluster is \emph{maximal} (denoted by $\Cmn N$), provided $N$ has the highest number of adjacent polygons in a tessellated surface (more than one maximal cluster in the same mesh is possible).  Similarly, a descriptive nucleus cluster 
is maximal (denoted by $\Cmdn N$), provided $N$ has the highest number of polygons in a tessellated surface descriptively near $N$, {\em i.e.}, the description of each $A\in \Cmdn N$ matches the description of nucleus $N$ and the number of polygons descriptively near $N$ is maximal (again, more than one $\Cmdn N$ is possible in a Vorono\"{i} tessellation).\\
\vspace{3mm}

\begin{figure}[!ht]
\centering
\includegraphics[width=65mm]{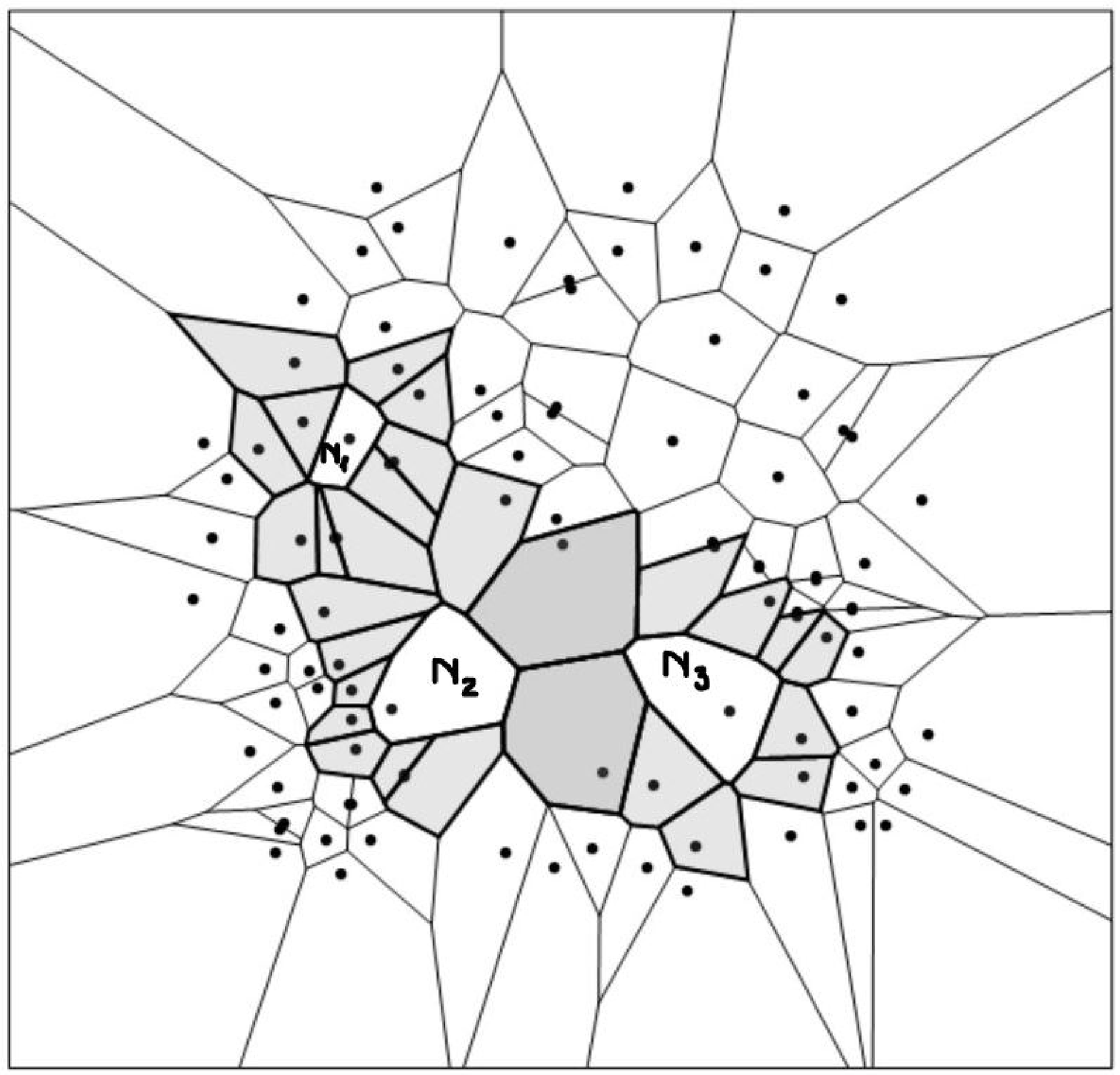}
\caption[]{ $\Cn\ N_1\ \sn\ \Cn\ N_2\ \mbox{and}\ \Cn\ N_2\ \sn\ \Cn\ N_3$}
\label{fig:N1N2N3maxClusters}
\end{figure}

\begin{example}
Let $X$ the collection of Vorono\"{i} regions shown in Fig.~\ref{fig:N1N2N3maxClusters} with $N_1,N_2,N_3\in X$. In addition, let $2^X$ be the family of all subsets of Vorono\"{i} regions in $X$.  Then $\Cn N_1,\Cn N_2,\Cn N_3\in 2^X$ nucleus clusters in the tessellation.
In this sample plane surface tessellation, $\Cn N_1\ \sn\ \Cn N_2$, since $A\ \sn\ B$ for some $A\in \Cn N_1, B\in \Cn N_2$.
Similarly, $\Cn N_2\ \sn\ \Cn N_3$.  In addition, nucleus clusters $\Cn N_2,\Cn N_3$ are maximal (denoted by $\Cmn N_2,\Cmn N_3$), since nuclei $N_2,N_3$ in the tessellation have the maximal number of adjacent Vorono\"{i} regions, namely, 10 adjacent regions.  Let the description of a nucleus cluster in the Euclidean plane be described by its number of sides of its nucleus.  Then $\Cmn N_2\ \snd\ \Cmn N_3$, since $N_2\ \snd\ N_3$, {\em i.e.}, the description of $N_2$ strongly matches the description of $N_3$ inasmuch as the description of the one nucleus is contained in the description of the other nucleus.  In a more complete description, we would also consider the gradient orientation of the nucleus edges. In the case, $N_2\ \snd\ N_3$, provided each nucleus has at least one edge with a gradient orientation that matches the gradient orientation of an edge in the other nucleus.  
\qquad \textcolor{blue}{\Squaresteel}
\end{example}

\begin{theorem}\label{thm:nucleusCluster}
Let $X$ be a set of Vorono\"{i} regions in the tessellation of a plane surface, endowed with the proximities $\sn,\snd,\dnear$ with $A,N\in X$.  In addition, let $2^X$ be the family of all subsets of Vorono\"{i} regions in $X$.  Then
\begin{compactenum}[1$^o$]
\item $A\in X$ implies $A\in \Cn N$ for some $N\in X$.
\item $\Cn N \in 2^X$ implies $A\ \sn\ N$ for some $A\in X$.
\item The union of all Vorono\"{i} nucleus clusters cover a plane surface, {\em i.e.},
\[
X = \mathop{\bigcup}\limits_{N\in X} \Cn N\ \mbox{(Nucleus cluster covering property)}.
\]
\item $N,N'\in X$ implies $\Cn N,\Cn N'\in 2^X$.
\item Let the description of $N\in X$ be the number of edges on the polygon $N$.  Then\\
$\Cmn N, \Cmn N'\in 2^X$ implies $N\ \snd\ N'$ for $N,N' \in X$.
\item\label{prop:sn} $\Cn N\ \sn\ \Cn N'$, if and only $A\ \sn\ B$ for some $A\in \Cn N, B\in \Cn N'$.
\item\label{prop:snd} $\Cn N\ \snd\ \Cn N'$, if and only if $A\ \snd\ B$ for some $A\in \Cn N, B\in \Cn N'$.
\item $A\ \sn\ \Cn\ B$, for $A\in X, \Cn\ B\in 2^X$ implies $\Cn N\ \sn\ \Cn B$ for some $N\in X$, where $A\ \sn\ N$.
\item $\Cn N\ \cap\ \Cn N'\neq \emptyset$ implies $A\ \sn\ B$ for some $A\in \Cn N, B\in \Cn N'$.
\item Let $\Cn N\ \dcap\ \Cn N' = \left\{A\in \Cn N \cup \Cn N': A\in \Phi(\Cn N)\ \&\ A\in \Phi(\Cn N')\right\}$ (descriptive intersection of nucleus clusters).  Then
$\Cn N\ \dcap\ \Cn N'\neq \emptyset$ implies $A\ \dnear\ B$ for some $A\in \Cn N, B\in \Cn N'$.
\end{compactenum}
\end{theorem}
\begin{proof} We prove only \ref{prop:sn}$^o$ and \ref{prop:snd}$^o$.  The proof of the remaining parts are direct consequences of the definitions.\\
\ref{prop:sn}$^o$:  $A\ \sn\ B$ ($A$ and $B$ have a common edge) for some $A\in \Cn N, B\in \Cn N'$, if and only if $\Cn N, \Cn N'$ are adjacent, if and only if $\Cn N\ \sn\ \Cn N'$.\\
\ref{prop:snd}$^o$: $A\ \snd\ B$ for some $A\in \Cn N, B\in \Cn N'$, if and only if the description of $A$ matches the description of $B$ ($A,B$ can be either adjacent or non-adjacent), if and only if $\Cn N\ \cap\ \Cn N'\neq \emptyset$, if and only if,  $A\ \snd\ B$.
\end{proof}

\begin{remark}
Let $X$ be a set of Vorono\"{i} regions in the tessellation of a plane surface, endowed with the proximities $\sn,\snd,\dnear$ with $A,N\in X$. 
From Theorem~\ref{thm:nucleusCluster}.\ref{prop:sn}, the nuclei in adjacent Vorono\"{i} nucleus clusters have a strong affinity in the sense that each of the clusters contains a Vorono\"{i} region that is strongly near a Vorono\"{i} region in an adjacent cluster.  For example, in Fig.~\ref{fig:N1N2N3maxClusters}, clusters $\Cn N_1,\Cn N_2$ share a pair of adjacent polygons.    The nuclei in adjacent Vorono\"{i} nucleus clusters have a strong descriptive affinity, provided the nuclei have matching descriptions.  It also the case that Vorono\"{i} regions $V(s),V(s')\in X$ are descriptively near, provided $s\ \dnear\ s'$, {\em i.e.}, the description of $s$ matches the description of $s'$.  Hence, from Theorem~\ref{thm:nucleusCluster}.\ref{prop:snd}, $\Cn V(s)\ \snd\ \Cn V(s')$.
\qquad \textcolor{blue}{\Squaresteel}
\end{remark}

\section{Main Results}

\begin{lemma}\label{lem:sn}
$A\ \sn \ B \Rightarrow A\ \snd \ B$.
\end{lemma}
\begin{proof}
$A\ \sn \ B$ implies that $A$ and $B$ have points in common.  Hence, there are points in $A$ and $B$ with the same descripitons, {\em i.e.}, $A\ \snd\ B$
\end{proof}

\begin{theorem}
$\Cn\ N\ \sn \ \Cn\ M \Rightarrow \Cn\ N\ \snd \ \Cn\ M$.
\end{theorem}
\begin{proof}
Immediate from Lemma~\ref{lem:sn} and Theorem~\ref{thm:nucleusCluster}.\ref{prop:snd}.
\end{proof}

 The descriptive intersection~\cite[\S 1.9, p. 43]{Peters2014book} of nonempty sets $A,B$ (denoted $A\ \ \dcap\ B$) in an $n$-dimensional Euclidean space $\mathbb{R}^n$ is defined in the following way.
\begin{description}
\item[{\rm\bf ($\boldsymbol{\Phi}$)}] $\Phi(A) = \left\{\Phi(x)\in\mathbb{R}^n: x\in A\right\}$, set of feature vectors.
\item[{\rm\bf ($\boldsymbol{\dcap}$)}]  $A\ \dcap\ B = \left\{x\in A\cup B: \Phi(x)\in \Phi(A)\ \&\ \Phi(x)\in \Phi(B)\right\}$.
\qquad \textcolor{blue}{$\blacksquare$}
\end{description}

That is, the descriptive intersection of $A$ and $B$ contains all $a\in A,b\in B$ that are descriptively near each other. 

\begin{theorem}
$\Cdn N\ \sn \ \Cdn M \Leftrightarrow A\ \dcap \ B\neq \emptyset$ for some $A\in \Csnd\ N, B\in \Csnd\ M$.
\end{theorem}
\begin{proof}
$\Cdn N\ \sn \ \Cdn M\Leftrightarrow \Cn N\ \snd\ \Cn M\ \mbox{(from the definition of $\snd$)}\ \Leftrightarrow A\ \snd\ B$ for some $A\in \Cn N, B\in \Cn M$ (from Theorem~\ref{thm:nucleusCluster}.\ref{prop:snd}), if and only if $A\ \dcap \ B\neq \emptyset$.
\end{proof}

\begin{definition}\label{def:Convexity}{Zelins'kyi-Soltan-Kay-Womble Convexity Structure}{\rm \cite{Soltan1984convexity,Zelinskii2015convexity,Kay1971convexitySpace}}.
Let $\mathscr{F} = 2^X$ be the family of all subsets of a nonempty set $X$ and let subfamilies $\mathscr{A},\mathscr{B}\in\mathscr{F}$.  The family $\mathscr{F}$ on $X$ is called a \emph{Zelins'kyi-Soltan-Kay-Womble convexity structure}, provided it satisfies the following axioms.
\begin{description}
\item[{\rm(C0)}] $\emptyset$ and $X$ belong to $\mathscr{F}$.
\item[{\rm(C1)}] $\mathscr{A}\ \cap\ \mathscr{B} \in \mathscr{F}$ for all subfamilies $\mathscr{A},\mathscr{B}\in\mathscr{F}$. 
 \qquad \textcolor{blue}{$\blacksquare$}
\end{description}
\end{definition} 

\noindent The pair $\left(X,\mathscr{F}\right)$ is a \emph{Zelins'kyi-Soltan-Kay-Womble convexity space}.

\begin{theorem}\label{thm:FamilyOfSetsConvexity}{\rm \cite{Peters2016ComputationalProximity}}
The family of all subsets $\mathscr{F} = 2^X$ of a nonempty set $X$ is a Zelins'kyi-Soltan-Kay-Womble convexity structure.
\end{theorem}
\begin{proof}
Let $\mathcal{A}\in \mathscr{F}$.  $X$ and $\emptyset$ are in $\mathscr{F}$.  In addition, $\mathop{\bigcap}\limits_{A\in\mathcal{A}} A\in \mathscr{F}$.  Hence, $\mathscr{F}$ is a Zelins'kyi-Soltan-Kay-Womble convexity structure.
\end{proof}

\begin{theorem}\label{thm:VoronoiConvexityStructure}
Let $X$ be a collection of Vorono\"{i} regions in the tessellation of a plane surface, $2^X$ the family of all subsets of $X$, $\Cn N,\Cn M\in 2^X$ such that $\Cn N\cap\Cn M\neq \emptyset$.  The family $2^X$ is a Zelins'kyi-Soltan-Kay-Womble convexity structure.
\end{theorem}
\begin{proof}
For a nonempty $X$, both $\emptyset$ and $X$ are subsets in $2^X$ (Axiom (C0)).  Let $\Cn N,\Cn M$ be subcollections in $2^X$.  $\Cn N\cap\Cn M\neq \emptyset$ implies that $\Cn N,\Cn M$ share at least one Vorono\"{i} region.  Consequently,  $\Cn N\cap\Cn M\in 2^X$ (Axiom (C1)).  Hence, from Theorem~\ref{thm:FamilyOfSetsConvexity}, $2^X$ is a Zelins'kyi-Soltan-Kay-Womble convexity structure.
\end{proof}

\begin{example}
From Theorem~\ref{thm:VoronoiConvexityStructure}, the collection of Vorono\"{i} regions $\left\{\Cn N_2,\Cn N_3\right\}$ in the tessellation shown in Fig.~\ref{fig:N1N2N3maxClusters} is a convexity structure, since $\Cn N_2,\Cn N_3$ have a Vorono\"{i} region in common.
\qquad \textcolor{blue}{$\blacksquare$}
\end{example}

\section{Applications}
\noindent Several applications arise from the introduction of Vorono\"{i} clustering.
\begin{description}
\item[{\bf Satellite Images}] Detecting surface objects and locations of sharp differences in terrain.  Surface objects are revealed by one or more occurrences of maximal nucleus clusters. 
\item[{\bf FMRI Images}]  High cortical activity corresponds to maximal nucleus clusters in brain tissue.  The leads to the detection and classification of cortical activity associated with the tessellation of fMRI images.  For example, the Vorono\"{i} mesh in Fig.~\ref{fig:N1N2N3maxClusters} has been extracted from the tessellation of an fMRI image of the brain.
Mesh nucleus clustering is directly related to recent studies of fMRI images~\cite{Tozzi2016CogNeuro4Dbrain,Tozzi2016JNeuroSciSymmetries}.
\item[{\bf Tomography Images}] High concentration of fossils correspond to the presence and distribution of maximal nucleus clusters in 3D tomography images derived from drill core samples.
\end{description}


\begin{thebibliography}{99}
%
%
%
\bibitem{DiConcilio2006} A. Di Concilio, G. Gerla, \emph{Quasi-metric spaces and point-free geometry}, Math. Structures Comput. Sci. {\bf 16} (2006), no. 1, 115–137, MR2220893.

\bibitem{DiConcilio2013mcs} A. Di Concilio, \emph{Point-free geometries: {P}roximities and quasi-metrics}, Math. in Comp. Sci. {\bf 7} (2013), no. 1, 31-42, MR3043916.
%
%
%
%
%
%
%

\bibitem{Edelsbrunner1999} H. Edelsbrunner, \emph{Computational Topology. Advances in discrete and computational geometry}. Contemp. Math., 223, Amer. Math. Soc., Providence, RI, 1999, MR1661380.

\bibitem{Edelsbrunner2001} H. Edelsbrunner, \emph{Geometry and Topology for Mesh Generation}. Cambridge University Press, Cambridge, UK, 2001, 2006. xii+177 pp. ISBN: 978-0-521-68207-7; 0-521-68207-X, MR2223897.

\bibitem{Edelsbrunner2014} H. Edelsbrunner, \emph{A Short Course in Computational Geometry and Topology}. Springer Briefs in Applied Sciences and Technology. Springer, Cham, 2014. x+110 pp. ISBN: 978-3-319-05956-3; 978-3-319-05957-0, MR3328629.

\bibitem{Edelsbrunner2010compTop} H. Edelsbrunner, \emph{Computational Topology.  An Introduction}. Amer. Math. Soc., Providence, RI, 2010. xii+241 pp. ISBN: 978-0-8218-4925-5, MR2572029.

%
%
%
%
\bibitem{Guadagni2015} C. Guadagni, \emph{Bornological Convergences on Local Proximity Spaces and $\omega_{\mu}$-Metric Spaces}, Ph.D thesis, Universit\`{a} degli Studi di Salerno, Dipartimento di Matematica, supervisor: A. Di Concilio, 2015, 72pp.

\bibitem{Inan2015} E. \.{I}nan, \emph{Algebraic Structures on Nearness Approximation Spaces}, Ph.D thesis, \.{I}n\"{o}n\"{u} University, Department of Mathematics, supervisors: S. Kele\c{s} and M.A. \"{O}zt\"{u}rk, 2015, vii+113pp.

\bibitem{Kay1971convexitySpace} D. Kay, E. Womble,\ {\it Automatic convexity theory and relationships between
the carath`eodory, helly and radon numbers}, Pacific Journal of Math. 38 (1971), no. 2, 471–485.

%
%
%
\bibitem{Naimpally1970} S.A. Naimpally, B.D. Warrack, \emph{Proximity spaces}, Cambridge University Press, Cambridge Tract in Mathematics and Mathematical Physics 59, Cambridge, UK, 1970, ISBN 978-0-521-09183-1, MR2573941.
%
%
\bibitem{Naimpally2009} S.A. Naimpally, \emph{Proximity Approach to Problems in Topology and Analysis}, Oldenbourg Verlag, Munich, Germany, 2009, 73 pp., ISBN 978-3-486-58917-7, MR2526304.

\bibitem{Naimpally2013} S.A. Naimpally, J.F. Peters, \emph{Topology with applications. Topological spaces via near and far. With a foreword by Iskander A. Taimanov}.  World Scientific Publishing Co. Pte. Ltd., Hackensack, NJ, 2013. xvi+277 pp. ISBN: 978-981-4407-65-6, MR3075111.


\bibitem{Peters2014book} J.F. Peters, \emph{Topology of Digital Images. Visual Pattern Discovery in Proximity
Spaces}. Intelligent Systems Reference Library 63, Springer (2014). DOI 10.1007/978-3-642-53845-2. URL
http://dx.doi.org/10.1007/978-3-642-53845-2. XV + 411 pp. Zentralblatt MATH Zbl 1295 68010.

\bibitem{Peters2016ComputationalProximity} J.F. Peters, \emph{Computational Proximity. Excursions in the Topology of Digital Images}. Springer, Intelligent Systems Reference Library, Berlin, 2016, \emph{in press}.

\bibitem{Peters2015visibility} J.F. Peters, \emph{Visibility in proximal {D}elaunay meshes}, Advances in Math. 4 (2015), no. 1, 41-47.

\bibitem{Peters2015AMSJmanifolds} J.F. Peters, C. Guadagni, \emph{Strong proximities on smooth manifolds and Vorono¨ı
diagrams}, Advances in Math. 4 (2015), no. 2, 97-107.

\bibitem{Peters2016relator} J.F. Peters, \emph{Proximal relator spaces}, Filomat (2016), \emph{in press}.

%

\bibitem{Soltan1984convexity} V.P. Soltan, \emph{Introduction to the axiomatic theory of convexity [{R}ussian With {E}nglish and {F}rench summariess}, Shtiintsa, Kishinev, 1984. 224 pp., MR0779643.

\bibitem{Tozzi2016CogNeuro4Dbrain} A. Tozzi, J.F. Peters, \emph{Towards a fourth spatial dimension of brain activity}, Cognitive Neurodynamics (2016), 1-11, DOI 10.1007/s11571-016-9379-z.

\bibitem{Tozzi2016JNeuroSciSymmetries} A. Tozzi, J.F. Peters, \emph{A topological approach unveils system invariances and broken symmetries in the brain}, J. of Neuroscience Research (2016), 1-12, \emph{in press}.

%
%
%
%

\bibitem{Zelinskii2015convexity} Y. Zelins'ky,\ {\it Generalized convex envelopes of sets and the problem of shadow}, Journal of Mathematical Sciences, 211 (2015), no. 5, 710-717.

\end{thebibliography}

\end{document}